\documentclass[11pt]{amsart}
\usepackage{amsfonts}
\usepackage{amsmath}
\usepackage{amssymb}
\usepackage{amsthm}
\usepackage{comment}
\usepackage{hyperref}

\newtheorem{thm}{Theorem}[section]
\newtheorem{lemma}[thm]{Lemma}
\newtheorem{prop}[thm]{Proposition}
\newtheorem{cor}[thm]{Corollary}
\newtheorem{example}[thm]{Example}

\theoremstyle{definition}
\newtheorem{defn}[thm]{Definition} 
\theoremstyle{remark}

\newtheorem{open}{Open Question}

\newtheorem*{xrem}{Remark}

\begin{document}


\title[Square Functions on $L^1$]{Square Functions and Variational Estimates for Ritt Operators on $L^1$}
 
\author[J. Hults]{Jennifer Hults}
\address{
Department of Mathematics and Statistics\\
University at Albany, SUNY, Albany, NY 12222 \\} 
\email{jhults@albany.edu} 

\author{Karin Reinhold-Larsson } 
\address{
Department of Mathematics and Statistics\\
University at Albany, SUNY, Albany, NY 12222 } 
\email{reinhold@albany.edu} 

\date{}


\keywords{Ritt operators, Square functions, Positive contractions, Convolution measures, Convolution operators, Variational estimates, Weak type (1,1)}


\begin{abstract}
A power bounded operator $T$ satisfying  $\sup_n n\lVert T^n-T^{n+1}\rVert<\infty$  is a Ritt operator. For such operators, we study the generalized square function
\[Q_{\alpha,s,r}^Tf=\Big( \sum_n n^{\alpha} |T^n(I-T)^rf|^s \Big)^{1/s}.\]
It is known that when $T$ is a positive contraction and a Ritt operator on $L^p$, $1<p<\infty$, then for  any integer $r\ge 1$, the square function $Q_{2r-1,2,r}^Tf$
 defines a bounded operator \cite{LeMX-Vq} on $L^p$. 
  In this work, we extend the theory to the endpoint case $p=1$. We show that if $T$ is a Ritt operator on $L^1$,
 then the generalized square function
$Q_{\alpha,s,r}^Tf $ 
is bounded on $L^1$ whenever $\alpha+1<sr$. 
In the particular setting where $T$ is a convolution operator of the form
$T_{\mu}=\sum_k \mu(k) U^kf$, with $\mu$ a probability measure on $\mathbb Z$ and  $U$ the composition operator induced by an invertible measure preserving transformation, we provide sufficient conditions on $\mu$ under which  $Q_{2r-1,2,r}^{T_{\mu}}f$ is of weak type (1,1), for $r>0$. 
We also establish bounds for variational  and oscillation norms, $\lVert n^{\beta} T^n(1-T)^r\rVert_{v(s)}$ and $\lVert n^{\beta} T^n(1-T)^r\rVert_{o(s)}$, for Ritt operators, highlighting endpoint behavior.
\end{abstract}

\maketitle


\section{Introduction}

Ritt operators have a long history, tracing back to the foundational work of Ritt \cite{Ritt} and later developed by Lyubich \cite{Lyub}, Nevanlinna and Zemanek \cite{Nev,NZ}, and others \cite{Bl}.
In recent years, they have been the focus of renewed interest \cite{A,ALeM,Bl,CCL,Dunn,GT,LeM-H,LeMX-max,LeMX-Vq,HH} due to the development of their $H^{\infty}$--functional calculus which made possible estimates for square functions in $L^p$ for $1<p<\infty$. 
Square functions and variational inequalities have been fruitful in martingale theory and harmonic analysis, with applications to Littlewood--Paley theory and in ergodic theory.  They were instrumental in bounding associated maximal functions, establishing convergence of sequences of operators and controlling their rate of convergence.

\begin{defn} 
Let $T$ be a linear operator on a Banach space such that it is (doubly) power--bounded: if supp$(\nu)\subset \mathbb N_0$, assume that $T$ is power--bounded, that is, $\sup_{n\ge 0} \lVert T^n\rVert<\infty$, otherwise assume  $T$ is invertible and doubly--power--bounded: $\sup_{n\in \mathbb Z} \lVert T^n\rVert<\infty$.
Let $\nu$ be a finite signed measure on $\mathbb Z$. Define the  operator induced by $\nu$ as
 \begin{equation}T_{\nu}f=\sum_k \nu(k) T^kf.\label{eq:Tmu} \end{equation}
\end{defn}

Given $0<a<1$, the series expansion of
$(1-x)^{a}=1-\sum_{k\ge 1} g(a,k) x^k$, $|x|\le 1$,
yields a probability measure $\nu_{a}$ on $\mathbb N$ with $\nu_{a}(1)=a$ and 
 \begin{equation}\nu_{a}(k)= g(a,k) =\frac{a |a-1| \ldots |a-k+1|}{k!} \ge 0, k>1. \label{eqn:valpha} \end{equation}

\begin{defn} 
Let $T$ be a power--bounded linear operator on a Banach space. For $0<a<1$ define $(I-T)^{a}$ 
as \[(I-T)^{a} = I - T_{\nu_{a}}= I - \sum_{k\ge 1} g(a,k) T^k.\]
\end{defn}

The operators $T_{\nu_{a}}$ have been considered by many authors. The first systematic study of their  properties  is due to Derrienic and Lin \cite{DerLin}.  

Decomposing any real $r>0$ into its integer and fractional components allows us to extend the definition of $(I-T)^r$ for all real $r>0$.

In particular, if $\mu$ is a finite measure on $\mathbb Z$,
$T_{\mu}^n(I-T_{\mu})^r f = T_{\nu_{n,r}} f$
where $\nu_{n,r}$ is a signed measure on $\mathbb Z$ satisfying $\hat \nu_{n,r}(t) = \hat \mu^n(t) (1-\hat\mu(t))^r$.

\begin{defn} 
For real numbers $s\ge 1,r\ge 0$, define generalized square functions as follows.
\[{\bf Q_{\alpha,s,r}}f ={\bf Q^{T}_{\alpha,s,r}}f = \Big(\sum_{n=1}^{\infty} n^{\alpha} |T^n(I-T)^r f|^{s} \Big)^{1/s}.\]
\end{defn}

\begin{defn} 
Let $T$ be a linear operator on a Banach space. $T$ is a Ritt operator if $T$ is power bounded and $\sup_n n \lVert T^n(1-T)\rVert<\infty$.
\end{defn}

For any sequence of complex number $\{x_n\}$, the s--variational norm is defined as
\[\lVert \{x_n\} \rVert_{v(s)} = \sup \Big(\sum_k |x_{n_k} - x_{n_{k+1}}|^s \Big)^{1/s}\]
where the sup is taken over all possible increasing sequences $\{n_k\}$.

Given any fixed sequence of increasing integers $\{n_k\}$, the s--oscillation norm is defined as
\[\lVert \{x_n\} \rVert_{o(s)} =\Big(\sum_k \max_{n_k\le n,m\le n_{k+1}} |x_n-x_m|^s \Big)^{1/s}.\]
Neither of these are proper norms, but they are seminorms. We refer the reader to \cite{JKRW,JSW} for discussion of their properties.

 The works of 
Le Merdy and Xu \cite{LeMX-Vq}, Arhancet and LeMerdy \cite{ALeM}, and Cuny, Cohen and Lin \cite{CCL}, established the following results for Ritt operators on $L^p$, $1<p<\infty$. 

\begin{thm} \label{thm:LeM} Let $(X,\mathcal B,m)$ be a $\sigma$--finite measure space, $1<p<\infty$, and $T$ a positive contraction of $L^p(X)$. 
If $\sup_n n \lVert T^n-T^{n+1} \rVert<\infty$, then, for any fixed real number $r> 0$, 
\begin{enumerate}
\item    ${\bf Q_{2r-1,s,r}}f$               
 is bounded on $L^p$ for $s\ge 2$;
\item for $s>2$, $\lVert \{T^nf\} \rVert_{v(s)}$ and, for any increasing sequence $\{n_k\}$, $\lVert \{T^nf\} \rVert_{o(2)}$, are bounded on $L^p$; and
\item $\lVert \{n^r T^n(I-T)^r f\} \rVert_{v(s)}$  is bounded on $L^p$ for $s\ge 2$.
\item Let $\{n_k\}\subset \mathbb N$ be an increasing sequence with $n_{k+1}-n_k \sim n_k$, then\\ $\left( \sum_k \max_{n_k\le n\le n_{k+1}} n^{2r}|T^n(I-T)^r|^2 \right)^{1/2}$ is bounded on $L^p$.
\end{enumerate}
\end{thm} 

Note: 
Le Merdy and Xu \cite{LeMX-Vq} (Proposition 4.1) proved part a) for integer $r$ and $s=2$ but since $(\sum_k |a_k|^s)^{1/s} \le (\sum_k |a_k|^2)^{1/2}$ if $s\ge 2$, a) also holds for $s>2$.
Arhancet and Le Merdy \cite{ALeM} subsequently showed that any pair of norms of the form $\lVert {\bf Q_{2r-1,2,r}}f\rVert_p$ (with ($r>0$) are equivalent in $L^p$, $1<p<\infty$. 
Hence a) holds for any $r>0$.
Additionally, Le Merdy and Xu \cite{LeMX-Vq} (Theorems 4.4 \& 5.6) proved b) and that,
 for $r\ge 1$ integer, $\lVert \{n^r T^n(I-T)^r f\} \rVert_{v(s)}$, for $s>2$; and $\lVert \{n^r T^n(I-T)^r f\} \rVert_{o(2)}$ are bounded on $L^p$.
Cuny, Cohen and Lin \cite{CCL} (Theorem 6.5) strengthened these results by establishing $\lVert \{n^r T^n(I-T)^r f\} \rVert_{v(2)}$ is  bounded on $L^p$, $r>0$. Their theorem focused on $0<r<1$ but their methods apply to any $r>0$. They also showed item d) for $n_k=2^k$ but their proof applies to sequences with $n_{k+1}-n_k \sim n_k$ with minor modifications.

In martingale theory, the equivalence between the square function and associate maximal functions is a classical result. 
This naturally raises the question of whether such an equivalence holds in other contexts. 
Cohen, Cuny and Lin  \cite{CCL}  established that a similar equivalence holds in the context of Ritt operators.

\begin{thm}\label{thm:CCL}
Let $(X,\mathcal B,m)$ be a $\sigma$-finite measure space, $1 <p <\infty$ and $T$ a positive contraction on $L^p(X, m)$. Then the following are equivalent:
\begin{enumerate}
\item  $\sup_n n\lVert T^n-T^{n+1}\rVert <\infty$,
\item there exists a constant $C_p>0$ such that $\lVert {\bf Q_{1,2,1}}f \rVert_p \le C_p \lVert f\rVert_p$.
\end{enumerate}
\end{thm}

Let $X=(X,\beta,m)$ be a $\sigma$--finite measure space. 
Throughout these notes, $Y$ denotes a Banach function space, that is, a Banach space of measurable complex--valued functions on $X$. In particular, this class includes $L^1(X)$ and Orlicz spaces.
%
\begin{thm} \label{thm:RittResult} Let 
$T$ be a Ritt operator on a Banach function space $Y$.  
Let $s\ge 1$ and $r> 0$, the following hold:
\begin{enumerate}
\item If $sr>\alpha+1$, then ${\bf Q^T_{\alpha,s,r}}f$ 
is bounded on $Y$; 
\item If $\beta<r$ and $\{n_k\}\subset \mathbb N$ an increasing sequence with $n_{k+1}-n_k\sim n_k$, then  
$\sup_n n^{\beta} |T^n (I- T)^{r} f|$ and
$\left(\sum_k  \sup_{n_k< n<n_{k+1}} n^{\beta s} |T^n (I- T)^{r} f|^s\right)^{1/s}$
are bounded on $Y$, 
 $\lim_{n\to\infty} n^{r} \lVert T^n (I- T)^{r} f\rVert=0$ and $\lim_{n\to \infty} n^{\beta} |T^n (I- T)^{r}f | = 0$ a.e.;
\item  If $\beta<r$ and $s\ge 1$ 
then both  $\lVert n^{\beta} T^n(I-T)^r f\rVert_{v(s)}$ and  $\lVert n^{\beta} T^n(I-T)^r f\rVert_{o(s)}$ are bounded on  $Y$;  
\item Let $\{n_k\}$ be any increasing sequence with $ n_{k+1} -n_k \sim n_k^{\gamma}$, for some $\gamma \in(0,1]$. 
If  $\beta< r+(1-\gamma)(1-1/s)$ %
then  
$\Big( \sum_k n_k^{\beta s}  \max_{ n_k\le n,m\le n_{k+1} } |(T^n -T^{m})(I-T)^r  f|^s \Big)^{1/s} $ and
$\Big( \sum_k n_k^{\beta s} |(T^{n_k}-T^{n_{k+1}}) (I-T)^r  f|^s \Big)^{1/s}$,
are bounded on  $Y$. In particular, these results holds for $\beta< r$ and $s\ge 1$ or $\beta\le r>0$ and $s> 1>\gamma$.  
\end{enumerate}
\end{thm}

\begin{xrem} 
By Theorem \ref{thm:CCL},  ${\bf Q_{1,2,1}}f$  is bounded on $L^p$ for any $p>1$, but on $L^1$, 
${\bf Q_{1,s,1}}f$ is bounded for $s>2$, and in general,  ${\bf Q_{r,s,r}}f$ is bounded on $L^1$ for $s>1+1/r$.
More generally, by Theorem \ref{thm:LeM}, ${\bf Q_{2r-1,2,r}}f$ is bounded on $L^p$ for any $p>1$ but on $L^1$,
 ${\bf Q_{2r-1,s,r}}f$ bounded  for any $s>2$,  and ${\bf Q_{\alpha,2,r}}f$ is bounded  for any $\alpha<2r-1$.
\end{xrem}

\begin{open} Is ${\bf Q_{\alpha,s,r}}f$ of weak type (1,1)  for $sr=\alpha+1$? 
\end{open}
\begin{open} For what range of $s$ are
  $\lVert  T^nf\rVert_{v(s)}$ and $\lVert  T^n f\rVert_{o(s)}$ of weak type (1,1)? And in general, for what range of $s$ are
  $\lVert  n^r T^n(I-T)^rf\rVert_{v(s)}$ and $\lVert n^r  T^n(I-T)^r f\rVert_{o(s)}$ of weak type (1,1)?
  
\end{open}

Estimates in $L^1$ present more challenges. For the maximal function $Mf=\sup_{n\ge 1} |T^nf|$ only partial results  are available, showing that $Mf$ is a weak (1,1) operator for a restricted class of operators acting on probability spaces 
\cite{BC,karin,Cuny-weak}. 
In what follows (and in Section 5), we provide answers for the above questions under assumptions analogous to those imposed in the study of $Mf$.

For the rest of this section $X=(X,\mathcal B,m)$ denotes a probability space and $\tau$ an invertible measure--preserving transformation on $X$. $T$ is the composition operator $Tf=f\circ \tau$.
To distinguish this case from the general setting, we denote
\begin{equation}
\tau_{\mu}f(x)= T_{\mu}f(x)=\sum_k \mu(k) f(\tau^k x).
\label{eq:1}
\end{equation}

We say that $\mu$ has {\bf bounded angular ratio} (BA) if there exist a constant $C>0$ such that
$|1-\hat\mu(t)|\le C (1-|\hat\mu(t)|)$ for all $|t|\le 1$, where $\hat{\mu}(t)=\sum_k \mu(k) e^{-2\pi i kt}$.

 Bellow and Calder\'on   \cite{BC}  showed that the maximal function $M_{\mu}f=
 \sup_{n\ge 1} |\tau^n_{\mu}f|$ is weak (1,1) 
 for centered measures $\mu$ with finite second moment. Such measures have the BA property.
 Further extensions of this result, relaxing the moment condition, were obtained by Dungey \cite{Dunn} (2011), Wedrychowicz  \cite{Chris} (2011), and Cuny \cite{Cuny-weak} (2016), 
 while Losert \cite{Losert1,Losert2} (1999, 2001) constructed measures $\mu$ without the BA property, for which pointwise convergence of $\tau^n_{\mu}$ failed.  

In \cite{Cuny-weak}, Cuny introduced a modification of the BA property, see properties $\text{BA}_1$ and $\text{BA}_2$ in Section 2, to obtain the following $L^1$ results.

\begin{thm} \label{thm:weak} 
\cite{Cuny-weak} 
Let $(X,\mathcal B,m)$ be a probability space and $\tau:X\to X$ an invertible measure preserving transformation.
Let $\mu$ be a probability measure on $\mathbb Z$ with property $\text{BA}_1$. Then 
\begin{enumerate}
\item $\sup_n |\tau_{\mu}^n f|$ is of weak type (1,1) and 
\item  $\tau_{\mu}$ is weak-$L^1$ Ritt: for $r\in \mathbb N$,   
$ m(x\in X: \sup_n n^r |(\tau_{\mu}^n-\tau_{\mu}^{n+r})f(x)|>\lambda) \le \frac{C_r}{\lambda} \lVert f\rVert_1$
for any $f\in L^1(X)$.
\item If  $\mu$ satisfies $\text{BA}_2$, then $\tau_{\mu}$ is Ritt on $L^1$.
\end{enumerate}
\end{thm} 
  
Strengthening the control over the fluctuations of $\tau_{\mu}^n$, we establish weak (1,1) results for generalized square--functions and variation--type operators.

\begin{thm} \label{thm:RittResult.mp} Let $(X,\mathcal B,m)$ be a standard probability space, $\tau:X\to X$ an invertible measure preserving transformation. Let $\mu$ be a probability measure on $\mathbb Z$ with property $\text{BA}_2$. Then, for any $r\ge  0$,
\begin{enumerate}
\item[A.] 
If $sr\ge \alpha+1$ and ${\bf Q^{\tau_{\mu}}_{\alpha,s,r}}f$ is bounded on $L^2$, then 
it is of weak type (1,1); 
\item[B.] Let $s> 1$ and $0\le \beta\le r$. If $ \lVert  n^{\beta} \tau_{\mu}^n (I-\tau_{\mu})^r f \rVert_{v(s)}$  is bounded on $L^2$, then it is of weak type (1,1),
and if $\lVert n^{\beta}  \tau_{\mu}^n (I-\tau_{\mu})^r  f\rVert_{o(s)}$ is bounded on $L^2$, then it is of weak type (1,1).
\end{enumerate}
\end{thm}

\begin{thm} \label{thm:main2}
Let $\tau, \mu$  as in Theorem \ref{thm:RittResult.mp} and $r> 0$. 
Then
\begin{enumerate}
\item  
${\bf Q^{\tau_{\mu}}_{s(r-1/2),s,r}}f$ is of weak type (1,1) for all $s\ge 2$. 
\item 
 $\lVert  \tau_{\mu}^nf\rVert_{v(s)}$ is of weak type (1,1) for all $s>2$; and $\lVert  \tau_{\mu}^n f\rVert_{o(2)}$ is of weak type (1,1).
 \item 
 More generally, 
  $\lVert n^{r} \tau_{\mu}^n (I-\tau_{\mu})^r  f\rVert_{v(s)}$ is of weak type (1,1) for all $s\ge 2$; and $\lVert n^{r}  \tau_{\mu}^n  (I-\tau_{\mu})^r  f\rVert_{o(2)}$ is of weak type (1,1).
 \item 
 Let $\{n_k\}\subset \mathbb N$ be an increasing sequence with $n_{k+1}-n_k \sim n_k$, then\\
 $\left( \sum_k \max_{n_k\le n\le n_{k+1}} n^{sr}|\tau_{\mu}^n(I-\tau_{\mu})^rf|^s \right)^{1/s}$ 
 is of weak type (1,1) for $s\ge 2$. 
\end{enumerate}
\end{thm}
Items a), b) and c) of Theorem \ref{thm:main2} are a straightforward application of Theorems \ref{thm:LeM} and \ref{thm:RittResult.mp}. Item d) is proven in Section 4. 
\medskip

 The article is organized as follows. Section 2, introduces the general framework and develops the main results for the special case of $T_{\mu}$.
 Section 3 presents a key technical lemma, central to the proofs of the general case. In Section 4 we provide detailed proofs for operators of the form $T_{\mu}$, and conclude with the proof of Theorem \ref{thm:RittResult}. Section 5, establishes weak type (1,1) bounds in Theorem \ref{thm:main2}.

\section{The case of  $T_{\mu}$} 

 We begin this section with a brief discussion of sufficient conditions under which $T_{\mu}$ is a Ritt operator on $L^1$. 
 For an $L^1-L^{\infty}$ contraction, Blunkck \cite{Bl} showed, via a clever interpolation argument, that if $T$ is Ritt on $L^2$, it is also Ritt on $L^p$, $1<p<\infty$. This interpolation technique does not extend to the case $p=1$. However, for the special case of $T_{\mu}$,
 Dungey  \cite{Dunn} showed $T_{\mu}$ is Ritt on $L^1$ provided $\mu$ satisfies a set of stronger regularity conditions than BA.
 
 \begin{thm} (Dungey \cite{Dunn}, Theorem 4.1) \label{thm:D}
Let $\mu$ be a probability measure on $\mathbb N_0$ such that there exists constants $0<\alpha,c<1$ with (i) $\mbox{Re} \, \hat\mu(t) \le 1-c |t|^{\alpha}$, and 
(ii)  $|\hat\mu'(t)|\lesssim |t|^{\alpha-1}$ for $0<|t|\le 1/2$.
If $T\in \mathcal(L^1)$ is power--bounded, then $T_{\mu}$ is Ritt on $L^1$.
\end{thm}

 Inspired by 
 \cite{Dunn} and also  \cite{Chris}, Cuny  \cite{Cuny-weak} considered broader spectral conditions, referred to below as properties $\text{BA}_1$ and $\text{BA}_2$, for measures supported on $\mathbb Z$ and showed these suffice to guarantee that $\tau_{\mu}$  is Ritt on $L^1$  in a broader setting.

\begin{defn} 
Let $\mu$ be a probability measure on $\mathbb Z$ whose Fourier transform $\hat\mu(t)$ 
is twice continuously differentiable on $|t|\in (0,1)$. We say $\mu$ has property {\bf $\text{BA}_1$} if there exists a continuous function $h(t)$ on $|t|\le 1$ with $h(0)=0$, $h(-t)=h(t)$, continuously differentiable on $|t|\in (0,1)$ satisfying the following conditions: there exists some constants $c,C>0$ such that
(i)  $|\hat\mu(t)|\le 1-c\, h(t)$,
(ii)  $|t \, \hat\mu'(t)|\le C\, h(t)$,
(iii)  $|\hat\mu'(t)|\le C\, h'(t)$, and
(iv)  $|t \, \hat\mu''(t)|\le C\, |\hat\mu'(t)|$. 
 If in addition $\hat\mu$ satisfies
(v) $h(t)\le C th'(t)$ for $0<t<1$, then we say $\mu$ satisfies condition {\bf $\text{BA}_2$}. 
\end{defn}

Note that  $\text{BA}_1$ (iii) implies bounded angular ratio (BA)
\begin{equation*} |1-\hat\mu(t)| \le | \int_0^t \hat\mu'(u) du | \lesssim \int_0^{|t|}  h'(u) du =  h(t) \lesssim 1-|\hat\mu(t)|.
\end{equation*}

Centered measures with finite second moment satisfy property $\text{BA}_2$ with $h(t)=t^2$. 
The next example exhibits a non--centered measure  without finite first moment that still satisfies $\text{BA}_2$.
\begin{example}   For fixed $0<a<1$, let $\nu_{a}$ be the probability measure on $\mathbb Z$ defined in (\ref{eqn:valpha}). 
Then $\hat\nu_{a}(t) = 1-(1-e^{2\pi i t})^{a}$. From \cite{Dunn}, $\nu_{a}$ satisfies the conditions of Theorem \ref{thm:D} and hence $T_{\nu_{a}} = I-(I-T)^{a}$ is a Ritt operator on $L^1$. Additionally, $\nu_{a}$ satisfies property $\text{BA}_2$.

\end{example}

More examples of measures with property $\text{BA}_2$ can be found in Cuny \cite{Cuny-weak}.

For general Banach function spaces, including $L^1$, we obtain the following results.

\begin{thm}\label{thm:main} 
Let $\mu$ be as in Theorem \ref{thm:weak} and $Y$ a Banach function space. 
Let $T$ be a (doubly) power--bounded linear operator on $Y$.
 If $r>0$ and  $sr>\alpha+1$, then ${\bf Q_{\alpha,s,r}^{T_{\mu}}}f$ is bounded  on $Y$. 
\end{thm}

\begin{cor} \label{cor:sup} 
Let $\mu$, $T$ and $Y$ as in Theorem \ref{thm:main} and $\{n_k\}\subset \mathbb N$ an increasing sequence with $n_{k+1}-n_k\sim n_k$.
If $0\le \beta<r$ and $s\ge 1$,
$\sup_n n^{\beta} |T_{\mu}^n (I- T_{\mu})^{r}f|$ and \\ $\Big( \sum_k \max_{n_{k}\le n\le n_{k+1}} n^{\beta s} |T_{\mu}^{n}(I-T_{\mu})^{r} f|^s\Big)^{1/s} $
are both bounded on $Y$, 
$\lim_{n\to \infty} n^{r} \lVert T_{\mu}^n (I- T_{\mu})^{r}f \rVert = 0$ and
 $\lim_{n\to \infty} n^{\beta} |T_{\mu}^n (I- T_{\mu})^{r}f | = 0$
 a.e..
\end{cor}

\begin{prop} \label{prop:longvar}
Let $\mu$, $T$ and $Y$ as in Theorem \ref{thm:main}.
Let  $\{n_k\}$ be an increasing sequence with $n_{k+1} -n_k \sim n_k^{\gamma}$ 
for some $0<\gamma\le 1$. Let $s\ge 1$, $r\ge 0$. 
If  $\beta<r+(1-\gamma)(1-1/s)$, 
 then 
$\Big( \sum_k n_k^{\beta s}  \max_{ n_k\le n,m\le n_{k+1} } |(T_{\mu}^n -T_{\mu}^{m})(I-T_{\mu})^r f|^s \Big)^{1/s} $ and  
$\Big( \sum_k n_k^{\beta s} |(T_{\mu}^{n_k}-T_{\mu}^{n_{k+1}})(I-T_{\mu})^rf|^s \Big)^{1/s}$ are bounded on $Y$. 
\end{prop}

\begin{prop} \label{prop:genvar}
Let $\mu$, $T$ and $Y$ as in Theorem \ref{thm:main}.
Let $r> 0$.
 If  $\beta< r$ and $s\ge 1$,        
then both   $\lVert n^{\beta} T_{\mu}^n(I-T_{\mu})^r f \rVert_{v(s)}$ 
 and
 $\lVert n^{\beta} T_{\mu}^n(I-T_{\mu})^r f \rVert_{o(s)}$ are bounded on $Y$.

\end{prop}


\section{Auxiliary Lemma} 
Across these notes, $c$ and $C$ denote constants whose values may change from one instance to the next, and $e(x)=e^{2\pi i x}$. For  $0\le x,y$, we say $x \lesssim y$ if there exists a constant $c>0$ such that $x \le c y$, and $x\sim y$ if $x \lesssim y$ and $y \lesssim x$.

For completeness, we recall the Stein--Minkowski integral inequality \cite{stein_M}:
Suppose that $(X_1,\nu_1)$ and $(X_2,\nu_2)$ are two $\sigma$--finite measure spaces and $g:X_1\times X_2\to \mathbb R$
is measurable. Then, for any $1\le s <\infty$,
\[\left[ \int_{X_2}
\left| \int _{X_1} g(x_1,x_2)\, d\nu_{1}(x_1) \right|^{s} d\nu_{2}(x_2)\right]^{1/s}~\leq 
\int _{X_1}\left(\int _{X_2}|g(x_1,x_2)|^{s}\,d\nu_{2}(x_2)\right)^{1/s} d\nu_1(x_1).\]

\begin{lemma}\label{lem:basic} 
Let $S=\{\Delta_n\}$ be a sequence of (finite) signed measures on $\mathbb Z$ whose Fourier transforms $\hat\Delta_n$ 
are twice continuously differentiable on $|t|\in (0,1)$. Let $T$ be a linear (doubly) power--bounded operator on a Banach function space, and
let
\begin{align*} 
A= & \int_{|t|<1/2} \frac 1{|t|} \Big( \sum_n |\hat \Delta_n(t)|^s \Big)^{1/s} dt,  \quad
& B=  \int_{|t|<1/2} |t| \Big( \sum_n |\hat \Delta''_n(t)|^s \Big)^{1/s} dt, \\
 C=& \sum_{k\neq 0} \frac 1{|k|} \Big( \sum_n |\hat \Delta_n(1/|2k|)|^s \Big)^{1/s},  \quad
 & D= \sum_{k\neq 0} \frac 1{|k|^2} \Big( \sum_n |\hat \Delta'_n(1/|2k|)|^s \Big)^{1/s}. 
\end{align*}
If $A, B, C$ and $D$ are all finite,  then, for any $f\in Y$,
$\Big\lVert  \Big( \sum_n |T_{\Delta_n}f|^s \Big)^{1/s} \Big\rVert \lesssim \lVert f\rVert.$
\end{lemma}

\begin{proof}
Without loss of generality, we assume $\sup_{n\ge 0} \lVert T^n \rVert = 1$ if $\cup_n \mbox{supp}(\Delta_n) \subset \mathbb N_0$ and  assume $\sup_{n\in \mathbb Z} \lVert T^n \rVert = 1$ otherwise.

\begin{align*}   \Big\lVert \Big( \sum_n |T_{\Delta_n}f|^s \Big)^{1/s} \Big\rVert
= & \Big\lVert \Big(\sum_n  \Big| \sum_k \Delta_n(k) T^kf\Big|^{s} \Big)^{1/s} \Big\rVert \\
\lesssim &   \Big\lVert   \Big( \sum_n  \Big|  \sum_{k\neq 0} \int_{|t|<1/|2k| } \hat \Delta_n(t) e(kt) dt  \,  T^kf \Big|^{s} \Big)^{1/s} \Big\rVert \\
 & +  \Big\lVert    \Big( \sum_n  \Big|  \sum_{k\neq 0}   \int_{1/|2k|<|t|<1/2} \hat \Delta_n(t) e(kt) dt  \,  T^kf \Big|^{s} \Big)^{1/s} \Big\rVert \\
& +     \Big( \sum_n  \Big|   \Delta_n(0)   \Big|^{s} \Big)^{1/s} \Big\lVert f \Big\rVert \\
 =  &  \mbox{I} + \mbox{II}+   \Big( \sum_n  \big|   \Delta_n(0)   \big|^{s} \Big)^{1/s} \Big\lVert f \Big\rVert.
\end{align*}

Applying the Stein--Minkowski integral inequality, we obtain the following estimate 
for the first term.
\begin{align*}
  \mbox{I} 
=  & \Big\lVert   \Big( \sum_n  \Big|  \sum_{k\neq 0} \int_{|t|<1/|2k|} \hat \Delta_n(t) e(kt) dt  \,  T^kf\Big|^{s} \Big)^{1/s} \Big\rVert \\
\le &  \lVert f\rVert \sum_{k\neq 0} \Big( \sum_n \Big|   \int_{|t|<1/ |2k| } \hat \Delta_n(t) e(kt) dt   \Big|^{s} \Big)^{1/s}\\
\le &   \lVert f\rVert \int_{|t|<1/2} \frac1{|t|}  \Big( \sum_n  \Big| \hat \Delta_n(t) \Big|^s \Big)^{1/s}  \, dt
= A \lVert f\rVert.
  \end{align*}
  Also 
\[ \Big( \sum_n  \big|   \Delta_n(0)   \big|^{s} \Big)^{1/s} \le  \int_{|t|<1/2} \Big( \sum_n  \Big| \hat \Delta_n(t) \Big|^s \Big)^{1/s}  \, dt
\le A.\]

For the second term, by the Stein--Minkowski integral inequality, we have
\begin{align*}
\mbox{II} \le & \Big\lVert   \Big( \sum_n  \Big|  \sum_{k\neq 0}   \int_{1/|2k| <|t|<1/2} \hat \Delta_n(t) e(kt) dt  \  T^kf \Big|^{s} \Big)^{1/s} \Big\rVert \\
\le & \Big\lVert  \sum_{k\neq 0}  \Big( \sum_n  \Big |  \int_{1/|2k|<|t|<1/2} \hat \Delta_n(t) e(kt) dt \Big |^s \Big)^{1/s} |T^kf | \Big\rVert\\
\le & \lVert f\rVert  \sum_{k\neq 0}  \Big( \sum_n  \Big |  \int_{1/|2k|<|t|<1/2} \hat \Delta_n(t) e(kt) dt \Big |^s \Big)^{1/s}.
\end{align*}

Since $\Delta_n(t), \Delta'_n(t)$ and $e(kt)$ are 1-periodic, $\Delta_n(-t)=\overline{\Delta_n(t)}$, and\\ $\Delta'_n(-t)=\overline{\Delta'_n(t)}$, we estimate
\begin{align}
\Big| \int_{1/|2k| <|t|<1/2} & \hat \Delta_n(t) e(kt) dt \Big| \le  
 \Big|\int_{1/|2k| <|t|<1/2} \hat \Delta'_n(t) \frac{e(kt)}{2\pi k} dt \Big| \nonumber \\
& + \Big|\frac{\hat \Delta_n(1/|2k| ) e(k/|2k| )}{2\pi k} - \frac{\hat \Delta_n(-1/|2k| ) e(-k/|2k| )}{2\pi k} \Big| \label{eq:1stderivative}\\   
\lesssim &  \Big| \int_{1/|2k| <|t|<1/2} \hat \Delta''_n(t) \frac{e(kt)}{4\pi^2 k^2} dt \Big| 
  + \Big| \frac{\hat \Delta_n(1/|2k| ) }{2\pi k} \Big| 
  +  \Big| \frac{\hat \Delta'_n(1/|2k| )}{4\pi^2 k^2}\Big|. 
 \nonumber
 \end{align}

By the Stein--Minkowski integral inequality,
\begin{align*}
\sum_{k\neq 0}   \Big( \sum_n  \Big |  \int_{1/|2k|<|t|<1/2} \hat \Delta''_n(t) \frac{e(kt)}{4\pi^2 k^2}  dt \Big |^s \Big)^{1/s} 
\lesssim  & \sum_{k\neq 0} \frac 1{k^2}  \int_{1/|2k|<|t|<1/2} \Big( \sum_n  \Big | \hat \Delta''_n(t) \Big |^s \Big)^{1/s} \ dt  \\
\lesssim & \int_{0<|t|<1/2}  |t| \,  \Big( \sum_n  \Big | \hat \Delta''_n(t) \Big |^s \Big)^{1/s}\, dt  = B .
\end{align*}

Thus, $\mbox{II}\lesssim (B+C+D) \lVert f\rVert$.

\end{proof}

The proof of this lemma can be adapted to the following setting.

\begin{lemma}\label{lem:basic2} 
Let $\Delta_n$ and $T$ as in Lemma \ref{lem:basic}, $\{n_k\}$ an increasing sequence, and \\$J_k=[n_k,n_{k+1})$.
Let
\begin{align*}
A= & \int_{|t|<1/2} \frac 1{|t|}        \Big( \sum_k n_k^{\beta} \max_{n\in J_k} |\hat \Delta_{n}(t)|^s \Big)^{1/s} dt,  \quad
& B=  \int_{|t|<1/2} |t|     \Big( \sum_k n_k^{\beta} \max_{n\in J_k} |\hat \Delta''_{n}(t)|^s \Big)^{1/s} dt, \\
C=&\sum_{l\neq 0} \frac 1{|l|}       \Big( \sum_k n_k^{\beta} \max_{n\in J_k} |\hat \Delta_{n}(1/|2l|)|^s \Big)^{1/s},  \quad
& D= \sum_{l\neq 0} \frac 1{|l|^2}       \Big( \sum_k n_k^{\beta} \max_{n\in J_k} |\hat \Delta'_{n}(1/|2l|)|^s \Big)^{1/s}. 
\end{align*}
If $A, B, C$ and  $D$ are all finite, then, for any $f\in Y$,
$ \Big\lVert    \Big( \sum_k n_k^{\beta} \max_{n\in I_k} |T_{\Delta_{n}}f|^s \Big)^{1/s} \Big\rVert \lesssim  \lVert f \rVert.$
\end{lemma}

\section{Proofs of Results for $T_{\mu}$ and Theorem \ref{thm:RittResult}}

\noindent \textit{Proof of Theorem \ref{thm:main}:}\\
Let $\Delta_n$ be
the measure on the integers defined by $\hat\Delta_n=n^{\alpha/s} \hat\mu^n (1-\hat\mu)^r$, that is, 
$T_{\Delta_n}=n^{\alpha/s} T_{\mu}^n (I-T_{\mu})^r$. Using Lemma \ref{lem:basic} with $\Delta_n$, it suffices to show that the corresponding terms defined in the lemma, A, B, C, and D, are bounded.

We assume $\alpha\ge -1$ because, if $\alpha<-1$, ${\bf Q_{\alpha,s,r}}f \le  {\bf Q_{-1,s,r}}f$.

If $\alpha>-1$, dominating the sum by an integral and using property $\text{BA}_1$ (i), we obtain
\[\sum_n n^{\alpha}  |\hat\mu(t)|^{ns} \le \sum_n n^{\alpha}  (1-c \, h(t))^{ns}    \lesssim   \frac 1{h(t)^{\alpha+1}}.\]
If $\alpha=-1$, 
\[\sum_n \frac 1n  |\hat\mu(t)|^{ns} \le \sum_n \frac 1n  (1-ch(t))^{ns}= | \ln(1-(1-ch(t))^s) | 
 \lesssim \frac 1{h(t)^{\gamma}}\]
for any $\gamma>0$.

For term A, using properties $\text{BA}_2$, we estimate
\begin{align}
A=\int_{|t|<1/2} \frac1{|t|} &  \Big( \sum_n  \Big| \hat\Delta_n(t) \Big|^s \Big)^{1/s} \ dt \nonumber \\
= &    \int_{|t|<1/2}  \Big( \sum_n n^{\alpha} | \hat\mu(t)|^{ns} \Big)^{1/s} \frac{|1-\hat\mu(t)|^{r}}{|t|} \ dt \nonumber \\
\lesssim  & \begin{cases} \int_{0<t<1/2} h(t)^{r-(\alpha+1)/s-1}  \, h'(t) \, dt & \mbox{ for } \alpha>-1 \label{eq:A} \\ 
 \int_{0<t<1/2} h(t)^{r-\gamma/s-1}  \, h'(t) \, dt & \mbox{ for } \alpha=-1. 
 \end{cases}
  \end{align}
When $\alpha>-1$, the integral is finite for $sr >\alpha+1$, and when $\alpha=-1$, the integral is finite for $r>0$ 
since we can choose $\gamma$ arbitrarily small, say $\gamma=sr/2$.

With $B_n=T^n (I-T)^r$,
\begin{align}
|\hat B'_n(t)| = | n \hat\mu^{n-1}(t) (\hat \mu'(t)) (1-\hat\mu(t))^r 
     - \hat\mu^{n}(t) \hat \mu'(t)(1-\hat\mu(t))^{r-1}|, \label{eq:Bprime} 
\end{align}
and by $\text{BA}_1$,
\begin{align*}
|\hat B''_n(t)| = & \big| n(n-1) \hat\mu^{n-2}(t) (\hat \mu'(t))^2 (1-\hat\mu(t))^r \\
   &  +n \hat\mu^{n-1}(t) \hat \mu''(t)(1-\hat\mu(t))^r 
 - 2n r \hat\mu^{n-1}(t) (\hat \mu'(t))^2 (1-\hat\mu(t))^{r-1}  \\
&+ r (r-1) \hat\mu^{n}(t) (1-\hat\mu(t))^{r-2}  (\hat\mu'(t))^2
 - r \hat\mu^{n}(t) (1-\hat\mu(t))^{r-1}  \hat \mu''(t) \big| \\
\lesssim & \Big[ n^2 |\hat\mu^{n-2}(t)| h^{2}(t) + 
n |\hat\mu^{n-1}(t)| h(t) 
 + |\hat\mu^{n}(t)|   \Big]  \frac{h^{r-1}(t)|\hat\mu'(t)|}{|t|}.
\end{align*}
Thus, for $\alpha>-1$,
\begin{align}
\Big(\sum_n   |\hat \Delta''_n (t)|^s\Big)^{1/s}= & \Big( \sum_n n^{\alpha}  |\hat B''_n (t)|^s\Big)^{1/s}  
\lesssim    
h(t)^{(r-1)-(\alpha+1)/s}  \frac{|\hat\mu'(t)|}{|t|}.\label{eq:B} 
\end{align}
Since $(1+\alpha)< s r $,
\begin{align*}
B= &  \int_{|t|<1/2} |t| \, \Big( \sum_n   |\hat\Delta''_n(t)|^s  \Big)^{1/s} dt   
 \lesssim   \int_{0<t<1/2}   h(t)^{(r-1)-(\alpha+1)/s } \, h'(t) \, dt <\infty.   
\end{align*} 

When $\alpha=-1$, choosing $0<\gamma<r$, the estimate is
\[   \int_{|t|<1/2} |t| \, \Big( \sum_n   |\hat\Delta''_n(t)|^s  \Big)^{1/s} dt   
 \lesssim   \int_{0<t<1/2}   h(t)^{r-\gamma-1 } \, h'(t) \, dt <\infty.  \]

For the remaining terms, we address the case $\alpha>-1$ since the estimates for the case $\alpha=-1$ follow similar arguments.

For $(1+\alpha)< s r $, by $\text{BA}_2$ and dominating the $k$th term by an integral over intervals $(\frac 1{k+1},\frac 1k)$, we have
\begin{align*}
 C= & \sum_{k\neq 0} \frac 1{|k|}
\Big(    \sum_n   |\hat\Delta_n(1/|2k| )|^s  \Big)^{1/s} \nonumber \\
\lesssim &  \sum_{k\neq 0} \frac 1{|k|} \Big(    \sum_n n^{\alpha} |\hat\mu(1/|2k|)|^{ns} 
|1-\hat\mu(1/|2k|)|^{s r}  \Big)^{1/s} \nonumber  \\
\lesssim & \sum_{k> 0} \frac 1k  h(1/2k)^{r - (\alpha+1)/s}
\lesssim c+ \int_0^{1/2} \frac{h(t)^{r - (\alpha+1)/s}}t \, dt   \nonumber \\
< & c+ \int_0^{1/2} h(t)^{r - (\alpha+1)/s-1} h'(t)\, dt <\infty.
\end{align*}

From (\ref{eq:Bprime}) and $\text{BA}_1$(ii) 
\begin{equation} \Big( \sum_n  |\hat\Delta'_n(t )|^s\Big)^{1/s} = 
\Big( \sum_n n^{\alpha}  |\hat B'_n(t )|^s\Big)^{1/s}
 \lesssim  h(t)^{r-1-(\alpha+1)/s} |\mu'(t)| \lesssim \frac{h(t)^{r-(\alpha+1)/s}}{|t|}.\label{eqn:C}
 \end{equation}
Thus
\begin{align*}
D = & \sum_{k\neq 0} \frac 1{|k|^2}
\Big(    \sum_n   |\hat\Delta'_n(1/|2k| )|^s  \Big)^{1/s} 
   \lesssim      \sum_{k\neq 0} \frac 1{|k|}   h(1/2k)^{r - (\alpha+1)/s}  <\infty.  
\end{align*}~\hfill$\square$
  \bigskip
  
  \noindent \textit{ Proof of Corollary \ref{cor:sup}:}\\
  Let $\{n_k\}$ be an increasing sequence in $\mathbb N$ such that $n_{k+1}-n_k\sim n_k$.  Then we have $n_{k+1}=n_{k+1}-n_k+n_k\lesssim 2n_k$.

Applying Abel's summation 
\begin{align*}
\sum_{j=1}^u j^{\beta } T_{\mu}^{j-1}(1-T_{\mu})f = & \sum_{j=1}^u j^{\beta } T_{\mu}^{j-1}- \sum_{j=1}^u j^{\beta } T_{\mu}^jf 
=  \sum_{j=0}^{u-1} (j+1)^{\beta } T_{\mu}^{j}f- \sum_{j=1}^u j^{\beta } T_{\mu}^jf\\
= &  \sum_{j=0}^{u-1} ((j+1)^{\beta } - j^{\beta } )T_{\mu}^{j}f- u^{\beta } T_{\mu}^uf\\
= &  \sum_{j=1}^{u} (j^{\beta } - (j-1)^{\beta } )T_{\mu}^{j-1}f- u^{\beta } T_{\mu}^uf\\
\end{align*}
Thus
\begin{align}
u^{\beta }  |T_{\mu}^{u} f| 
\le & \sum_{j=1}^{u} (j^{\beta }-(j-1)^{\beta } )|T_{\mu}^{j-1} f| 
 + \sum_{j=1}^u j^{\beta } |T_{\mu}^{j-1}(1-T_{\mu})f| \nonumber \\
\lesssim & \sum_{j=1}^{u} j^{\beta  -1} |T_{\mu}^{j-1}  f| 
 + \sum_{j=1}^{u} j^{\beta } |T_{\mu}^{j-1}(1-T_{\mu})f|.\label{eqn:abel}
\end{align}
Decomposing $n\in(n_k,n_{k+1}]$ as $n=n_{k-1}+l$ with $l\in (n_k-n_{k-1},n_{k+1}-n_{k-1}]$, and applying the above to the function $T_{\mu}^{n_{k-1}}(1-T_{\mu})^rf$ with $u=l$,
\begin{align*}
n^{\beta } & |T_{\mu}^{n}(1-T_{\mu})^r f| \lesssim \frac{n^{\beta}}{l^{\beta}}\left[
\sum_{j=1}^{l} j^{\beta  -1} |T_{\mu}^{j+n_{k-1}-1} (1-T_{\mu})^r f| 
 + \sum_{j=1}^{l} j^{\beta} |T_{\mu}^{j+n_{k-1}-1}(1-T_{\mu})^{r+1}f| \right] \\
\lesssim & \sum_{j=n_{k-1}}^{n_{k+1}-1} j^{\beta  -1} |T_{\mu}^j (1-T_{\mu})^r f| 
 + \sum_{j=n_{k-1}}^{n_{k+1}-1} j^{\beta} |T_{\mu}^{j} (1-T_{\mu})^{r+1} f |
  \end{align*}
  because $\frac{n}{l} \le \frac{n_{k+1}}{n_k-n_{k-1}}\lesssim \frac{2n_k}{n_{k-1}}\lesssim  4$.
  Thus 
 \begin{align*}
 \left( \sum_k \right. & \left. \max_{n_{k}<n\le n_{k+1}}  n^{\beta s}|T_{\mu}^{n}(I-T_{\mu})^{r} f|^s \right)^{1/s}\le 
 \sum_k  \max_{n_{k}<n\le n_{k+1}}  n^{\beta}|T_{\mu}^{n}(I-T_{\mu})^{r} f| \nonumber \\
 \lesssim & \sum_k \sum_{j=n_{k-1}}^{n_{k+1}-1} j^{\beta  -1} |T_{\mu}^j (1-T_{\mu})^r f| 
 + \sum_k \sum_{j=n_{k-1}}^{n_{k+1}-1} j^{\beta} |T_{\mu}^j (1-T_{\mu})^{r+1} f |. \nonumber \\
\lesssim &  \ {\bf Q_{\beta-1,1,r}}f + {\bf Q_{\beta,1,r+1}}f.  
\end{align*}
By Theorem \ref{thm:main},
 both generalized square functions are bounded in $Y$ when $0\le \beta< r$. Therefore 
$\sup_n n^{\beta } |T_{\mu}^n(I - T_{\mu})^{r})f|$ is also bounded in $Y$.

By Corollary 6.2 of \cite{CCL}, $\lim_{n\to \infty} n^{r} \lVert T_{\mu}^n(I - T_{\mu})^{r}f \rVert = 0$.

From these results, it follows $n^{\beta} |T_{\mu}^n(I - T_{\mu})^{r}f|\to 0$ a.e. as well.
~\hfill$\square$ \bigskip\\

Note: In  \cite{CCL} Proposition 6.4, it was shown that 
\begin{equation*}
\Big( \sum_k \max_{n_{k}<n\le n_{k+1}} n^{2\beta} |T_{\mu}^{n}(I-T_{\mu})^{r} f|^2\Big)^{1/2} \lesssim 
{\bf Q_{2\beta -1,2,r}}f + {\bf Q_{2(\beta+1)-1 ,2,r+1}}f.
\end{equation*}
This result holds for general $s>1$. 
Let  $s'=s/(s-1)$.
From equation (\ref{eqn:abel}) with $\beta=1$,
\begin{align*}
 |T_{\mu}^{u} f| 
\lesssim & \frac 1u\sum_{j=1}^{u}  |T_{\mu}^{j-1}  f| 
 + \frac 1u \sum_{j=1}^{u} j |T_{\mu}^{j-1}(1-T_{\mu})f|\\
 \le & \left( \frac 1u\sum_{j=1}^{u}  |T_{\mu}^{j-1}  f|^s \right)^{1/s} +  
  \left( \sum_{j=1}^{u} j^{s/s'} |T_{\mu}^{j-1}(1-T_{\mu})f|^s\right)^{1/s} \  \frac 1u \ \left(\sum_{j=1}^{u} j^{s'/s} \right)^{1/s'}.
\end{align*}
Since $s/s'=s-1$ and $\sum_{j=1}^{u} j^{1/(s-1)} \lesssim u^{1+1/(s-1)}=u^{s'}$, 
\[ |T_{\mu}^{u} f| \lesssim \left( \frac 1u\sum_{j=1}^{u}  |T_{\mu}^{j-1}  f|^s \right)^{1/s} +  
  \left( \sum_{j=1}^{u} j^{s-1} |T_{\mu}^{j-1}(1-T_{\mu})f|^s\right)^{1/s}.\]
Following similar arguments as in the corollary, it follows
\begin{equation}
\Big( \sum_k \max_{n_{k}<n\le n_{k+1}} n^{s\beta} |T_{\mu}^{n}(I-T_{\mu})^{r} f|^s\Big)^{1/s} \lesssim 
{\bf Q_{\beta s -1,s,r}}f + {\bf Q_{(\beta+1) s-1 ,s,r+1}}f.\label{eqn:summaxs}
\end{equation}

\medskip

\noindent \textit{ Proof of Proposition \ref{prop:longvar}:}\\
We prove only the result for the convolution measures with $\hat \Delta_{k,r}=n_k^{\beta} ({\hat\mu}^{n_k}  - {\hat\mu}^{n_{k+1}})(1-\hat \mu)^r$ because the other case follows similar arguments.
It suffices to verify that  $\{\Delta_{k,r}\}_k$ satisfy the conditions of Lemma \ref{lem:basic}.

If $\delta\ge 0$,
 \begin{equation*}\sum_k n_k^{s\delta} (n_{k+1}-n_k) |\hat\mu(t)|^{n_k s}\lesssim \frac1{h(t)^{s\delta+1}} . 
 \end{equation*} 

\begin{align}
\Big( \sum_k   |\hat\Delta_{k,r}(t)|^{s}  \Big)^{1/s}
\le  & \Big( \sum_k  n_k^{\beta s}  |\hat\mu(t)|^{sn_k} \, |1-\hat\mu(t)^{(n_{k+1}-n_k)}|^s \Big)^{1/s} |1-\hat\mu(t)|^r \nonumber \\
 \le &  \Big( \sum_k n_k^{\beta s} (n_{k+1}-n_k)^s |\hat\mu(t)|^{s n_k}  \Big)^{1/s}  |1-\hat\mu(t)|^{r+1} \nonumber \\
 \lesssim & \Big( \sum_k n_k^{\beta s+\gamma (s-1)} (n_{k+1}-n_k) |\hat\mu(t)|^{s n_k}  \Big)^{1/s}  |h(t)|^{r+1}   \nonumber \\
 \lesssim & \frac{h(t)^{r+1}}{ h(t)^{\beta+\gamma+(1-\gamma)/s}}. \label{eq:nk} 
\end{align}
Under the assumptions for $\gamma$, $\beta$ and $s$ we have
$\beta+\gamma+(1-\gamma)/s <1+r$.

\noindent Then, since $|h(t)|\lesssim |t| h'(t)$,
\begin{align*}
 \mbox{A} = &
\int_{|t|<1/2} \frac 1{|t|} \Big( \sum_k  |\hat\Delta_{k,r}(t)|^{s}  \Big)^{1/s}  dt
\lesssim & \int_{|t|<1/2}  \frac{h(t)^{r} \, h'(t)}{ h(t)^{\beta+\gamma+(1-\gamma)/s}}  dt =K <\infty,
  \end{align*}
and dominating the terms of the sum by an integral over intervals of length $l^{-2}$,
\begin{align*}
\mbox{C} =&  \sum_{l\neq 0} \frac 1{|l|}
\Big(    \sum_k |\hat\Delta_{k,r}(1/|2l| )|^s  \Big)^{1/s} 
\lesssim    \sum_{l\neq 0} \frac {h(1/|2l|)^{1+r-(\gamma+\beta)-(1-\gamma)/s}}{|l|} \lesssim K <\infty. 
\end{align*}
Let $B_k=\hat\mu^{n_k}(t)-  \hat\mu^{n_{k+1}}(t)$, then
\begin{align*}
| B'_k(t) |
\le &
 n_k | \hat \mu^{n_k-1}(t)| \,  
\Big| 1-\hat\mu(t)^{n_{k+1}-n_k} \Big| \,|\hat\mu'(t) | 
 +  (n_{k+1}-n_k)  \Big| \hat\mu(t) \Big|^{n_{k+1}-1}  \, |\hat\mu'(t)|\\
 \lesssim & n_k^{1+\gamma} | \hat \mu^{n_k-1}(t)| \, |\hat\mu'(t)| |1-\hat\mu(t)| 
 + n_k^{\gamma} | \hat \mu^{n_{k+1}-1}(t)| \, |\hat\mu'(t)|,
 \end{align*}
 and
\begin{align*}
 B''_k
  = &  n_k(n_k-1) \hat\mu^{n_k-2}(t) (\hat \mu'(t))^2 - n_{k+1}(n_{k+1}-1) 
         \hat\mu^{n_{k+1}-2}(t)(\hat \mu'(t))^2  \nonumber \\
 & +n_k \hat\mu^{n_k-1}(t) \hat \mu''(t) - n_{k+1} \hat\mu^{n_{k+1}-1}(t) \hat \mu''(t)  \\ 
= & n_k(n_k-1)\hat\mu^{n_k-2}(t) (1-\hat\mu^{n_{k+1}-n_k}(t)) (\hat \mu'(t))^2 \nonumber\\
& +[n_k(n_k-1)-n_{k+1} (n_{k+1}-1)] \hat\mu^{n_{k+1}-2} (\hat \mu'(t))^2 \nonumber\\
& +  n_k  \hat\mu^{n_k-1}(t) (1-\hat\mu^{n_{k+1}-n_k}(t)) \hat \mu''(t) \nonumber\\
& +(n_k-n_{k+1}) \hat\mu^{n_{k+1}-1}(t) \hat \mu''(t).\nonumber
\end{align*}
Since
\[ \hat\Delta'_{k,r} = n_k^{\beta} \big[ B'_k(t) (1-\hat\mu(t))^r + B_k(t) r (1-\hat\mu(t))^{r-1} \hat\mu'(t)\big],\]
 \begin{equation} \Big( \sum_k  |\hat\Delta'_{k,r}|^s \Big)^{1/s}
 \lesssim  \frac {h^r(t) |\hat\mu'(t)| }{h(t)^{(\gamma+\beta)+(1-\gamma)/s }}. \label{eq:nkprime} 
 \end{equation}
Estimating
\begin{align*}
|n_k(n_k-1)-n_{k+1}(n_{k+1}-1)| 
\lesssim  n_k^{\gamma}(n_{k+1}+n_k+1)  
\lesssim n_k^{1+\gamma},
\end{align*}
we have
\begin{align*}
|n^{\beta} B_k''(t)&(1-\hat\mu(t))^r| \lesssim  \Big[   n_k^{2+\gamma+\beta}  \hat\mu^{n_k-2}(t) h(t)^2 
     + n_k^{1+\gamma+\beta} |\hat\mu(t)|^{n_{k+1}-1}  h(t)  \Big. \\
 & \Big. + n_{k}^{1+\gamma+\beta} |\hat\mu(t)|^{n_{k+1}-2}  h(t) 
 +n_{k}^{\gamma+\beta} |\hat\mu(t)|^{n_{k+1}-1}  \Big] \frac{h(t)^r |\hat\mu'(t)|}{|t|}.
\end{align*}
  Since 
  \begin{align*}  
  |\hat\Delta''_{k,r}| \sim & \quad n_k^{\beta} \Big[   |B''_k(t)| |1-\hat\mu(t)|^r + |B'_k(t)| |1-\hat\mu(t)|^{r-1} |\hat\mu(t)| \\
  & +
  |B_k(t)|  |1-\hat\mu(t)|^{r-2} |\hat\mu'(t)|^2 + |B_k(t)|  |1-\hat\mu(t)|^{r-1} |\hat\mu''(t)|
 \Big],
  \end{align*}
  we have
\begin{align}
\Big(   \sum_k   |\hat\Delta''_{k,r}(t)|^s    \Big)^{1/s} 
  \lesssim    \frac{h(t)^r}{h(t)^{(\gamma+\beta)+(1-\gamma)/s } }   \frac{|\hat\mu'(t)|} {|t|}, \label{eq:nk2primev2}
\end{align}
and
\begin{equation*}
\mbox{B} =  \int_{|t|<1/2} |t|\,  \Big( \sum_k  |\hat\Delta''_{k,r}(t)|^s  \Big)^{1/s} \, dt 
  \lesssim   \int_{0<t<1/2} 
\frac {h(t)^r \ h'(t)}{h(t)^{(\gamma +\beta)+(1-\gamma)/s }}\, dt<\infty.
\end{equation*}
%
%
For the last term, from equation (\ref{eq:nkprime}) and $|t\hat\mu'(t)|\lesssim h(t)$,  we have
\[
\mbox{D} \lesssim        \sum_{l\neq0} \frac 1{l^2} \frac {h(1/|2l|)^r |\hat \mu'(1/|2l|)|}{h(1/|2l|)^{(\gamma+\beta)+(1-\gamma)/s }} 
 \lesssim \sum_{l\neq 0}  \frac {h(1/|2l|)^{r +1 -(\gamma+\beta)-(1-\gamma)/s}}{|l|} \lesssim K
<\infty. 
\]

\hfill$\square$

\bigskip
\noindent \textit{ Proof of Proposition \ref{prop:genvar}:} \\ 
Let $\Delta_{n,r}f=T_{\mu}^n(I-T_{\mu})^r f$, and 
$\{n_k\}$ any increasing sequence. 
By Abel summation,
\[\sum_{u=n}^m u^{\beta} T_{\mu}^u(I-T_{\mu})f=\sum_{u=n+1}^{m+1} (u^{\beta} - (u-1)^{\beta})T_{\mu}^uf + n^{\beta}T_{\mu}^n -(m+1)^{\beta} T_{\mu}^{m+1}f.\]
Hence 
\[\left| n_k^{\beta}T_{\mu}^{n_k}f -n_{k+1}^{\beta} T_{\mu}^{n_{k+1}}f \right| \lesssim  \left|   \sum_{u=n_k}^{n_{k+1}-1} u^{\beta} T_{\mu}^u(I-T_{\mu})f \right| +
\left| \sum_{u=n_k+1}^{n_{k+1}} (u^{\beta} - (u-1)^{\beta})T_{\mu}^uf  \right|.\]
Applied to the function $(I-T_{\mu})^r f$ we obtain
\begin{align}
|D_{k,\beta}f| = &  |n_k^{\beta} \Delta_{n_k,r} f - n_{k+1}^{\beta} \Delta_{n_{k+1},r}f | 
\lesssim   \sum_{u=n_k+1}^{n_{k+1}} u^{\beta-1} | \Delta_{u,r}f| +  \sum_{u=n_k}^{n_{k+1}-1}  u^{\beta} |\Delta_{u,r+1}f|. \label{eq:abel}
\end{align}
Since $s\ge 1$,
\begin{align*}
\Big( \sum_k |D_{k,\beta}f|^s \Big)^{1/s} \le  \sum_k |D_{k,\beta}f| \lesssim &  \sum_u  u^{\beta-1 } |\Delta_{u,r}f|  + \sum_u u^{\beta}  |\Delta_{u,r+1}f|.
\end{align*}
Then
\[ \lVert n^{\beta}\Delta_{n,r}f \rVert_{v(s)}  \lesssim {\bf Q_{\beta-1 ,1,r}}f + {\bf Q_{\beta ,1,r+1}}f,\]
which, by Theorem \ref{thm:main}, are bounded on $Y$ because $\beta<r$.

Since  $ \lVert n^{\beta} \Delta_{n,r}f \rVert_{o(s)}\le   \lVert n^{\beta} \Delta_{n,r}f \rVert_{v(s)}$ the result also holds for oscillation norms.
~\hfill$\square$
\medskip

\noindent \textit{ Proof of Theorem \ref{thm:RittResult}.}\\
By Theorem 1.3. in \cite{Dunn}, there exists  a power bounded 
operator $R$ and $a\in (0,1)$ 
such that 
$T=I-(I-R)^{a} = R_{\nu_{a}}$, where $\nu_{a}$ is the probability measure on $\mathbb Z$ defined in equation (\ref{eqn:valpha}).  
Therefore Theorem \ref{thm:main}, Corollary \ref{cor:sup}, Propositions \ref{prop:longvar}  and \ref{prop:genvar} apply to $T$.
\hfill$\square$


\section{The case of  $\tau_{\mu}$}
We now turn to the special case of the operators $\tau_{\mu}$
induced by a measure preserving transformation on a probability space.

By the ergodic decomposition of probability measures,
 it suffices to work with ergodic systems.
 
  \begin{lemma}\label{lem:ergo_decom} 
Let $X=(X,\mathcal B,m)$ be a standard probability space and $\tau:X\to X$ a measure preserving transformation. 
Let $U_{\tau}$ be a sublinear operator defined on measurable functions on $X$.
Suppose that there exists a constant $C>0$ such that for any $\tau$--invariant ergodic probability measure $\nu$ on $(X, \mathcal B)$,
$\nu\left(|U_{\tau}f|>\lambda \right) \le \frac C{\lambda} \int |f| d\nu$.
Then 
\[m\left(|U_{\tau}f|>\lambda \right) \le \frac C{\lambda} \int |f| dm.\]
\end{lemma}

\begin{proof}
There are several versions of the ergodic decomposition of measures \cite{Walters, CSF}. 
We use the formulation based on conditional expectation on a standard probability space \cite{Sarig}. (The argument works for any other formulation.)

Let $Inv(\tau)=\{E \in \mathcal B : E = \tau^{-1}E\}$. Let $m_y$ be the conditional probabilities with respect to $Inv(\tau)$. (Refer to \cite{Sarig} to see how the $\{m_y\}$ are constructed.)
Then, for $m$-a.e. $y\in X$, $m_y$ is $\tau$--invariant and ergodic and 
$\int_X f dm= E(E(f|Inv(\tau))=  \int_X \int_X f dm_y dm(y)$.
Then
\begin{align*}
m(|U_{\tau}f|>\lambda) 
 = & \int_X m_y\left(|U_{\tau}f|>\lambda \right) dm(y) 
\le  \int_X \frac C{\lambda} \int_X |f| dm_y \ dm(y)=\frac C{\lambda} \int_X |f|\ dm. 
\end{align*}
\end{proof}

For operators $\tau_{\mu}$, the estimates developed in Sections 4 can be improved with the aid of the Ergodic Calder\'on--Zygmund  decomposition introduced by R. Jones \cite{Roger}, because  the cancellation provided by the "bad" function gives improved control on Fourier coefficients.

\begin{prop}\label{prop:Roger_CZ}  \cite{Roger} Ergodic CZ--Decomposition.\\ 
Let $(X, m)$ be a probability space and $\tau : X \to X$ and ergodic, invertible measure preserving transformation. 
Let $f\in L^1(X)$ and\\ $f^*(x)= \sup_{n,m} \frac 1{n+m+1} |\sum_{k=-m}^n f(\tau^kx)|$, the maximal function for averages. Then $f$ can be decomposed as $f=g+b$,  with $g\in L^{\infty}$ and $b\in L^1(X)$, called the "bad" function,  with the following properties. 
\begin{enumerate}
\item[a)]
The set $B=\{x:  f^*(x) \le \lambda\}$ decomposes as $B= \cup_i B_i$ where the $B_i$'s are pairwise disjoint.
\item[b)] $b=\sum_i b_i$ where each $b_i$ is supported on a set $E_i$ of the form $E_i=\cup_{j=1}^{i} \tau^jB_i$, union of pairwise disjoint sets.
\item[c)] $\sum_{k=1}^{i}b_i(\tau^kx)=0$ if $x\in B_i$.
\item[d)] $\sum_{k=1}^{i}|b_i(\tau^kx)| \le 2 \lambda$ if $x\in B_i$.
\item[e)] $\sum_i m(E_i) \le 2 \frac{\lVert f\rVert_1}{\lambda}$.
\item[f)] $\lVert g\rVert_{\infty} \le 2\lambda$ and $\lVert g\lVert_1\le \lVert f\rVert_1$. 
\end{enumerate}
\end{prop}

This decomposition reduces the problem on $L^1$ to working with simpler building blocks called ‘atoms,’ which are defined below.
\begin{defn} \label{def:atom} Let $\tau:X\to X$ be an invertible, measure preserving  transformation on a standard probability space $(X,\mathcal B,m)$. Given $\lambda>0$, 
a {\bf $\lambda$--atom} is a function $a\in L^1(X)$ for which there exist a measurable set $B\in X$ and an integer $d>0$ such that
\begin{enumerate}
\item  $\{\tau^k B\}_{k=1}^n$ are pairwise disjoint sets, 
\item $a$ is supported on the set $E=\cup_{k=1}^d \tau^k B$,
\item $\sum_{k=1}^d a(\tau^k x) =0$ for all $x\in B$,
\item $\sum_{k=1}^d |a(\tau^k x)|\le 2\lambda$ if $x\in B$.
\end{enumerate}
We call $B$ the generator set (of length $d$) and define  
\[E^*=E\cup \tau^{d}E \cup \tau^{2d}E  \cup \tau^{-d} E \cup \tau^{-2d} E.\]
\end{defn}
Notice that, as a consequence of the properties of atoms,
\begin{equation}\label{eq:def}
\|a\|_1 =  \int_E |a| \, dm =  \sum_{k=1}^d \int_{\tau^k B} |a| \, dm 
=  \sum_{k=1}^d \int_{B} |a\circ \tau^k| \, dm \le 2\lambda \, m(B).\end{equation}

\begin{prop} \label{prop:atoms}
Let  $X$ and $\tau$ as in Definition \ref{def:atom}. 
Suppose $Q$ is a sublinear operator defined on measurable functions on $X$ satisfying
\begin{enumerate}
\item $Q$ is bounded on $L^2(X)$,
\item there exists a constant $C>0$ such that, for every $\lambda$--atom $a$,
\[\int_{(E^*)^c} |Q(a)| dm \le C \, \lambda m(E),\] where 
   $E$ and $E^*$ are defined as in Definition \ref{def:atom}.
\end{enumerate}
Then $Q$ is of weak type (1,1). 
\end{prop}

\begin{proof}
By Lemma \ref{lem:ergo_decom}, it suffices to condider the case when $(X, \mathcal B, m,\tau)$ is an ergodic measure preserving system.

Given $\lambda>0$ and $f\in L^1$, decompose $f$  via the ergodic--CZ decomposition, with the sets $B,B_i, E_i$ 
and the associated good and bad functions $g$ and $b$ defined as in Proposition \ref{prop:Roger_CZ}.

Since $Q$ is sublinear, we have
\[
m(x:|Q f(x)|>2\lambda)\le  m(x:|Q g(x)|>\lambda) + m(x:|Q b(x)|>\lambda).\]
By the  $L^2$--boundedness of $Q$,
 $ m(x:|Q g(x)|>\lambda) \lesssim \frac{\lVert g\rVert_2^2}{\lambda^2} \le  \frac{2\lVert g\rVert_1}{\lambda}  
\le  \frac{2\lVert f\rVert_1}{\lambda}.
$ 

Let $E_i^*=E_i\cup \tau^{i}E_i \cup \tau^{2 i}E_i  \cup \tau^{-i} E_i \cup \tau^{-2i} E_i$.
By property e) if the ergodic-CZ decomposition,
\begin{align*}
m(x:|Q b(x)|> \lambda)\le  &
m(\cup_i E^*_i)+
m(x\in (\cup_i E^*_i)^{c}: |Q b(x)|>\lambda)\\
\le & 5 \sum_i m(E_i) + \frac 1{\lambda} \int_{(\cup_i E^*_i)^{*c}} |Q b|\, dm \\
\le & 10 \frac{\lVert f\rVert_1}{\lambda} +  \frac 1{\lambda} \sum_i \int_{(E^*_i)^{*c}} |Q b_i|\, dm. 
\end{align*}
Noting that each $b_i$ is a $\lambda$--atom, apply assumption b) 
to obtain
\[ \int_{(\cup_i E^*)^c} |Q b(x)| dm \le \sum_i  \int_{E_i^{*c}} |Q b_i(x)| dm \le C \lambda \sum_i  m(E_i) \le 2 C \|f\|_1,\]
completing the proof.
\end{proof}

\begin{thm}\label{thm:ergodic_key} 
Let $\tau:X\to X$ be an invertible, measure preserving  transformation on a standard probability space $(X,\mathcal B,m)$. 
Let $s\ge 1$, $\{n_k\}$ an increasing sequence of positive integers, $J_k=[n_k,n_{k+1})$, 
and $\{\Delta_n\}$ a collection of finite, signed measures on $\mathbb Z$ such that $\{\hat \Delta_n\}$ are twice continuously differentiable on $|t|\in(0,1)$. Suppose $\{\Delta_n\}$ 
satisfy the following conditions:\\
$\int_{|t|<.5} \big( \sum_k  \max_{n\in J_k} |\hat \Delta_n(t)|^s \big)^{1/s} \, dt<\infty$, \\
$\sum_{|j|>1} \frac 1{j^2}  \big( \sum_k \max_{n\in J_k} |\hat \Delta_n(1/|j|)|^s \big)^{1/s}<\infty$,\\
$\int_{|t|<.5} \big( \sum_k \max_{n\in J_k} |\hat \Delta''_n(t)|^s \big)^{1/s}|t|^2 \, dt<\infty$, \\
and $\sum_{|j|>1} \frac 1{|j|^3}  \big( \sum_k \max_{n\in J_k} |\hat \Delta'_n(1/|j|)|^s \big)^{1/s}<\infty$.  \\
 Let \[Qf=\Big(\sum_k\max_{n\in J_k} |\tau_{\Delta_n} f |^s \Big)^{1/s}.\] Then
\begin{enumerate}
\item for every $\lambda$--atom $a$ ($\lambda>0$), 
\begin{equation*} \int_{(E^*)^c} Qa\ dm \le C \, \lambda m(E), 
\end{equation*}
where 
$E$ and $E^*$ are as in Definition \ref{def:atom}; 
\item if  $Qf$ is bounded on $L^2$
then $Qf$ is weak (1,1).
 \end{enumerate}
\end{thm}

\begin{proof}
Let $a$ be a $\lambda$--atom and let $B$, $d$, $E$ and $E^*$ as in Definition \ref{def:atom}.
Since $a$ is supported on $E$, and using the cancellation property c) of  Definition \ref{def:atom},
\begin{align*}\tau_{\Delta_n} a(x)=\sum_{j} \Delta_n(j) a(\tau^j x) = &
\sum_{j: \tau^jx\in B} \sum_{l=1}^{d} (\Delta_n(j+l) - \Delta_n(j)) a(\tau^{j+l} x).
\end{align*}
Then
\begin{align*}
(Q a(x))^s  =& 
\sum_{k=1}^{\infty} \max_{n\in J_k} \Big|
\sum_{j: \tau^jx\in B} \sum_{l=1}^{d} (\Delta_n(j+l) - \Delta_n(j)) a(\tau^{j+l} x)\Big|^s\\
=  & \sum_{k=1}^{\infty} \max_{n\in J_k} \left| \sum_{j: \tau^jx\in B} \sum_{l=1}^{d}a(\tau^{j+l}x) 
\int_{|t|<.5} 
\hat\Delta_n(t) (e((j+l)t)- e(jt) ) dt  \right|^s. 
\end{align*}
Let $I(x)$ and $II(x)$ be the expresions obtained by replacing the integral over $\{|t|<.5\}$ with integrals over  $\{|t|<1/|j|\}$ and $\{1/|j|<|t|<.5\}$ respectively. 
\begin{align*}
I(x)= & \left(\sum_{k=1}^{\infty} \max_{n\in J_k}
\left| \sum_{j: \tau^jx\in B} \sum_{l=1}^{d}a(\tau^{j+l}x)
\int_{|t|<1/|j|} 
\hat\Delta_n(t) (e((j+l)t)- e(jt) ) dt  \right|^s \right)^{1/s} \\
II(x)= & \left(\sum_{k=1}^{\infty} \max_{n\in J_k}
\left| \sum_{j: \tau^jx\in B} \sum_{l=1}^{d}a(\tau^{j+l}x)
\int_{1/|j|<|t|<.5} 
\hat\Delta_n(t) (e((j+l)t)- e(jt) ) dt  \right|^s \right)^{1/s}.
\end{align*}
By the triangular inequality of the $\ell^s$--norm,
\[Q a(x)\le  I(x)+II(x).\]

From the construction of $E^*$, if $x\in (E^*)^c$ 
 and $\tau^jx\in B$, we must have $|j|\ge 2d>1$.

Estimate for $I(x)$: 
\begin{align*} I(x)= &
\left(\sum_{k=1}^{\infty} \max_{n\in J_k}
\left| \sum_{j: \tau^jx\in B} \sum_{l=1}^{d}a(\tau^{j+l}x)
\int_{|t|<1/|j|} 
\hat\Delta_n(t) (e((j+l)t)- e(jt) ) dt  \right|^s \right)^{1/s}\\
\le  & 
\sum_{j: \tau^jx\in B} \sum_{l=1}^{d} |a(\tau^{j+l}x)| \int_{|t|<1/|j|}  
\left(\sum_{k=1}^{\infty}  \max_{n\in J_k} |\hat\Delta_n(t)|^s \right)^{1/s}  
|e((j+l)t)- e(jt) | dt  \\
\le  & 
\sum_{j: \tau^jx\in B} \sum_{l=1}^{d} |l| \ |a(\tau^{j+l}x)| \int_{|t|<1/|j|}  
\left(\sum_{k=1}^{\infty} \max_{n\in J_k}  |\hat\Delta_n(t)|^s \right)^{1/s}   |t| \ dt  \\
\le  & \  2\lambda d
\int_{|t|<.5}  
\left(\sum_{k=1}^{\infty} \max_{n\in J_k}  |\hat\Delta_n(t)|^s \right)^{1/s}   |t| \sum_{|j|>1} 1_{[x:x\in \tau^{-j}B]}(x) 1_{[j:|j|<1/|t|]}(j) \ dt . 
\end{align*}
Then, 
\begin{align*}  \int_{(E^*)^c}  I(x) dm 
\le & \ 2\lambda d
\int_{|t|<.5}    
\left(\sum_{k=1}^{\infty}  \max_{n\in J_k} |\hat\Delta_n(t)|^s \right)^{1/s}  |t| \sum_{1<|j|<1/|t|}  m((E^*)^c\cap \tau^{-j}B)  \\
\le & \ 4\lambda d \ m(B)
\int_{|t|<.5}  
\left(\sum_{k=1}^{\infty} \max_{n\in J_k}  |\hat\Delta_n(t)|^s \right)^{1/s}    \ dt   
\le  \ C \lambda m(E).
\end{align*}

Handling $II(x)$ requires integration by parts. We will need the following estimates.
Because $\Delta_n(t),\Delta'_n(t),e(kt)$ are 1-periodic, the same procedure as (\ref{eq:1stderivative})
 yields, for $|j|>2|l|$,
\begin{align*}
| \int_{|t|>1/|j|} &
\hat\Delta_n(t) (e((j+l)t)- e(jt) ) dt | \lesssim  \\
\lesssim &  \left| \int_{|t|>1/|j|} \hat\Delta'_n(t) (\frac{e((j+l)t)}{j+l}- \frac{e(jt)}{j} ) dt \right|
 + \left| \hat\Delta_n(1/|j|) \right| \frac{|l|}{j^2}.
 \end{align*}
 From \cite{BC}, since
$2|l|<2d<|j|$,  we have
\[|\frac{e((j+l)t)-1}{(j+l)^2}- \frac{e(jt)-1}{j^2} | \lesssim \frac{|l||t|}{|j|^2}. \]
Thus
\begin{align*}
\int_{|t|>1/|j|} &
\hat\Delta'_n(t) (\frac{e((j+l)t)}{(j+l)}- \frac{e(jt)}{j} ) dt\\
\lesssim &  \left| \int_{|t|>1/|j|} \hat\Delta''_n(t) (\frac{e((j+l)t)-1}{(j+l)^2}- \frac{e(jt)-1}{j^2} ) \ dt \right|\\
& + \left| \hat\Delta'_n(1/|j|) \right| \frac{|l|}{|j|^3} \\
\le &   \int_{|t|>1/|j|} |\hat\Delta''_n(t)| |t| \frac{|l|}{|j|^2} \ dt+  
\left| \hat\Delta'_n(1/|j|) \right| \frac{|l|}{|j|^3}.
\end{align*}

By Stein--Minkowski integral inequality and property d) of $\lambda$--atoms,
\begin{align*} 
II_{1}(x) = &  \left(\sum_{k=1}^{\infty} \max_{n\in J_k} \left| \sum_{j: \tau^jx\in B} \sum_{l=1}^{d}a(\tau^{j+l}x)
\int_{|t|>1/|j|} 
\hat\Delta_n(t) (e((j+l)t)- e(jt) ) dt  \right|^s \right)^{1/s} \\
\lesssim &
\left(\sum_{k=1}^{\infty} \max_{n\in J_k}  \left| \sum_{j: \tau^jx\in B} \sum_{l=1}^{d}  a(\tau^{j+l} x) 
\int_{1/|j|<|t|<.5} |\hat\Delta''_n(t)|  \frac{l |t|}{|j|^2} \ dt \right|^s\right)^{1/s} 
 \end{align*}
\begin{align*} 
\le &  \ d \sum_{j: \tau^jx\in B} \sum_{l=1}^{d}  |a(\tau^{j+l} x)|  \int_{1/|j|<|t|<.5} \left(\sum_{k=1}^{\infty}  \max_{n\in J_k} |\hat\Delta''_n(t)|^s\right)^{1/s} \frac{|t|}{|j|^2} \ dt\\
 \le & \ 2\lambda d \int_{|t|<.5} \left(\sum_{k=1}^{\infty} \max_{n\in J_k} |\hat\Delta''_n(t)|^s\right)^{1/s} |t|
 \sum_{|j|>1/|t|} 1_{[\tau^jx\in B]}(x) \frac{1}{|j|^2} \ dt.\\
 \end{align*}
 Integrating over $(E^*)^c$,
\begin{align*}  
\int_{(E^*)^c} II(x) dm   
  \le & \ 2\lambda d \int_{|t|<.5} \left(\sum_{k=1}^{\infty} \max_{n\in J_k} |\hat\Delta''_n(t)|^s\right)^{1/s} |t| \sum_{|j|>1/|t|} \frac{1}{|j|^2}
  m(B)\\
\lesssim & \  \lambda m(E) \int_{|t|<.5} \left(\sum_{k=1}^{\infty} \max_{n\in J_k} |\hat\Delta''_n(t)|^s\right)^{1/s} |t|^2 \ dt= C \lambda m(E).
\end{align*}
Lastly, we need to handle two remaining pieces.
\begin{align*} 
II_{2}(x) = & \int_{(E^*)^c} \left(\sum_k \max_{n\in J_k} \left| \sum_{j: \tau^jx\in B} \sum_{l=1}^{d} |a(\tau^{j+l} x)| 
 \left| \hat\Delta_n(1/|j|) \right| \frac{|l|}{|j+l|^2} \right|^s\right)^{1/s} dm\\
 \lesssim & \int_{(E^*)^c }   d \sum_{j: \tau^jx\in B} \sum_{l=1}^{d} |a(\tau^{j+l} x)| 
 \left(\sum_k \max_{n\in J_k} \left| \hat\Delta_n(1/|j|) \right|^s \right)^{1/s} \frac{1}{j^2} dm\\
  \lesssim & \  \lambda  m(E)  \sum_{|j|>1} 
 \frac 1{j^2}  \left(\sum_k  \max_{n\in J_k}  \left| \hat\Delta_n(1/|j|) \right|^s \right)^{1/s}
 \le  C \lambda m(E);
 \end{align*}
 and
\begin{align*} 
II_{3}(x) = & \int_{(E^*)^c } \left(\sum_k \max_{n\in J_k} \left| \sum_{j: \tau^jx\in B} \sum_{l=1}^{d} |a(\tau^{j+l} x)| 
 \left| \hat\Delta'_n(1/|j|) \right| \frac{|l|}{|j+l|^3} \right|^s\right)^{1/s} dm \\
  \lesssim & \int_{(E^*)^c }   d \sum_{j: \tau^jx\in B} \sum_{l=1}^{d} |a(\tau^{j+l} x)| 
 \left(\sum_k \max_{n\in J_k} 
 \left| \hat\Delta'_n(1/|j|) \right|^s \right)^{1/s} \frac{1}{|j|^3} dm\\
  \lesssim & \ \lambda m(B) d \sum_{|j|>1} \left(\sum_k \max_{n\in J_k} 
 \left| \hat\Delta'_n(1/|j|) \right|^s \right)^{1/s}
 \frac 1{|j|^3} 
   \le   
     C \lambda m(E).
 \end{align*}

Combining the estimates for $I(x),II_{1}(x), II_{2}(x)$ and $II_{3}(x)$, we obtain the desired result
\begin{align*} 
\int_{(\cup_i E^*)^c } Qa(x)\ dm
\le  \ C  \lambda \ m(E). 
\end{align*}

\end{proof}

\medskip
\noindent {\it Proof of Theorem \ref{thm:RittResult.mp}  }\\

\noindent {\it Proof of Part A:}\\
Assuming that ${\bf Q_{{\alpha,s,r}}}$ is bounded on $L^2$,
 we simply need to check the requirements in Theorem \ref{thm:ergodic_key} for $\hat \Delta_n=n^{\alpha/s} \hat\mu^n(1-\hat\mu)^{r}$ and the sequence $n_k=k$, $r\ge 0$.

First  note that, under the assumptions for $h$,   if $(1+\alpha)\le sr$,
$\int_{|t|<.5}  
h^{r-(1+\alpha)/s}(t) \, dt<\infty.$
Using estimates for $\Delta_n(t)$ in equation (\ref{eq:A}),  for $(1+\alpha)\le sr$ we have
\[\int_{|t|<.5}   \left(\sum_{n=1}^{\infty}   |\hat\Delta_n(t)|^s \right)^{1/s}  dt\le 
\int_{|t|<.5}  h^{r-(1+\alpha)/s}(t)dt <\infty,\]
and
\begin{align*}
\sum_{|j|>1}  \frac{1}{j^2}
\left(\sum_n  
 \left| \hat\Delta_n(1/|j|) \right|^s \right)^{1/s} 
\le &   \sum_{|j|>1}
 \frac 1{j^2} h(1/j)^{r-(1+\alpha)/s}
 \lesssim  \int_0^1 h(t)^{r-(\alpha+1)/s}dt <\infty.
 \end{align*}
 From property $\text{BA}_1$ (ii) and equation (\ref{eqn:C}),
\begin{align*}
\sum_{j>1} \frac{1}{|j|^3} \left(\sum_n 
 \left| \hat\Delta'_n(1/|j|) \right|^s \right)^{1/s}  \le &
 \sum_{j>1}
 \frac 1{j^2} h(1/j)^{r-(\alpha+1)/s} <\infty.
 \end{align*}
By the estimates  for $|\hat\Delta''_n(t)|$ in equation (\ref{eq:B}) and $\text{BA}_1$ (ii),
\begin{align*}
\int_{|t|<.5} \left(\sum_{n=1}^{\infty} |\hat\Delta''_n(t)|^s\right)^{1/s} |t|^2 dt \le &
\int_{|t|<.5}  h(t)^{(r-1)-(\alpha+1)/s}|\hat\mu'(t)| |t| dt \\
\le & \int_{|t|<.5}  h(t)^{r-(\alpha+1)/s}dt <\infty .
\end{align*}

\medskip

\noindent {\it Proof of Part B: }\\
 Since  $\lVert \{n^{\beta}\tau_{\mu}^n(I-\tau_{\mu})^r f\} \rVert_{v(s)}$ is bounded on $L^2$,  
 by Proposition \ref{prop:atoms} it suffices to check the conditions of 
  Theorem \ref{thm:ergodic_key} for the corresponding sublinear operator $Q$ defined by $\hat \Delta_n=n^{\beta}({\hat \mu}^n(1-\hat \mu)^r)$. 

Because $s>1$ we can choose  $0<\delta<\min(1,(s-1)/(s+1))$. Starting with the sequence $\{2^k\}$, enlarge  it by adding $N_k=2^{k(1-\delta)}$ equally spaced points in between, 
call these $\{r_{k,j}\}_{j=0,N_k-1}$. 
These points have the properties: $r_{k,0}=2^k$ and $r_{k,j+1}-r_{k,j}\sim 2^k/N_k=2^{\delta k}\sim r_{k,j}^{\delta}$.  
Denote the full collection of these as $\{r_j\}$ with the indices rearranged appropriately. 


Let  $B_{n,r}=n^{\beta} \tau_{\mu}^{n}(I-\tau_{\mu})^r$ and 
  $K_n=B_{r_j,r}$ for $r_j\le n<r_{j+1}$. 
  Also let  $I_j=[r_j,r_{j+1}]$.
\[ \lVert \{ B_{n,r}  a\} \rVert_{v(s)} \\
   \le 
 \lVert \{ B_{n,r} -K_n)   a\} \rVert_{v(s)}+\lVert \{K_n a\} \rVert_{v(s)} \]

 By properties of variational norms (\cite{JKRW,JR,JSW}), 
\begin{align*} 
\lVert \{ (B_{n,r} -K_n)   a\} \rVert_{v(s)}  
\lesssim  &  \Big( \sum_j   \lVert\{ (B_{n,r} -B_{r_j,r})   a: n\in I_j \} \rVert_{v(s)}^s  \Big)^{1/s}.    
\end{align*}

Using Abel summation as in equation (\ref{eq:abel}), bound the last term by
\begin{align*} 
 \lVert & \{ (B_{n,r} -B_{r_j,r})  a : n\in I_j\} \rVert^s_{v(s)} 
  \lesssim  
 \sum_{n\in I_j}    |(B_{n,r} -B_{r_j,r}) a |^s \\
& \lesssim   
 r_j^{\delta} \Big(\sum_{n=r_j+1}^{r_{j+1}}    n^{\beta-1} |\tau_{\mu}^{n} (1-\tau_{\mu})^r a|  + 
\sum_{n=r_j}^{r_{j+1}-1}   n^{\beta}  |\tau_{\mu}^{n}  (1-\tau_{\mu})^{r+1}  a|  \Big)^s \\
& \lesssim   r_j^{\delta(1+s)}  \sum_{n=r_j}^{r_{j+1}}  n^{(\beta-1)s}  |\tau_{\mu}^{n}  (1-\tau_{\mu})^r  a|^s   
+   r_j^{\delta(1+s)} \sum_{n=r_j}^{r_{j+1}-1}   n^{\beta s}  |\tau_{\mu}^{n}  (1-\tau_{\mu})^{r+1}  a|^s   .
  \end{align*}
Thus
\begin{align} \label{eq:B1}
 \lVert \{ (B_{n,r} -B_{r_j,r}) a\} \rVert_{v(s)}
  \lesssim &  {\bf Q_{(\beta-1)s+\delta(1+s),s,r}}a+ {\bf Q_{\beta s+\delta(1+s),s,r+1}}a.
 \end{align}
Since $\beta\le r$ and $\delta(1+s)+1<s$, these are bounded on $L^1$.

When $r>0$
\begin{align*} 
\lVert \{K_n  &   a\} \rVert_{v(s)} \lesssim    \sum_j  |  (r_j^{\beta} \tau_{\mu}^{r_j}-  r_{j+1}^{\beta} \tau_{\mu}^{r_{j+1}}) (1-\tau_{\mu})^r a| \\
 \lesssim  & \sum_j    r_j^{\beta} |(\tau_{\mu}^{r_j}-  \tau_{\mu}^{r_{j+1}}) (1-\tau_{\mu})^r a|
 + \sum_j   ( r_{j+1}^{\beta} -r_j^{\beta}) |\tau_{\mu}^{r_{j+1}} (1-\tau_{\mu})^r a|  \\
  \lesssim  & \sum_j    r_j^{\beta} |(\tau_{\mu}^{r_j}-  \tau_{\mu}^{r_{j+1}}) (1-\tau_{\mu})^r a|
 + \sum_j     r_{j+1}^{\beta-1+\delta}  |\tau_{\mu}^{r_{j+1}} (1-\tau_{\mu})^r a|.  \\
 \end{align*}
Since $\beta \le r$,  Lemma \ref{lem:rn_weak} part (a) applies to the first term with $s=1$, and 
Lemma \ref{lem:rn_weak} part (b) applies to the second term.

When $r=0$ and $\beta=0$, we have
\[ \lVert \{K_n    a\} \rVert_{v(s)} \lesssim    \sum_j  |  ( \tau_{\mu}^{r_j}-  \tau_{\mu}^{r_{j+1}}) a|. \]
Thus Lemma \ref{lem:rn_weak} part (a) applies. 

In either case, 
we obtain
\begin{align} \label{eq:B2} \int_{(E^*)^c}  \lVert \{K_n    a(x)\} \rVert_{v(s)} \, dm 
\lesssim  \ \lambda  m(E),
\end{align}

Finally, from equations (\ref{eq:def}), (\ref{eq:B1}) and (\ref{eq:B2})
\begin{align*}
\int_{(E^*)^c} \lVert \{ B_{n,r} a\} \rVert_{v(s)}  dm 
\le &
\int_{(E^*)^c}  \left[ \lVert \{ (B_{n,r} -K_n)  a\} \rVert_{v(s)}
  + 
\lVert \{K_n   a(x)\} \rVert_{v(s)}\right] \, dm\\
 \lesssim & \  \lambda m(E).
  \end{align*}

The result for the oscillation norm follows from noticing that $ \lVert B_{n,r}a \rVert_{o(s)}\le   \lVert B_{n,r}a \rVert_{v(s)}$.
\hfill$\square$

\noindent {\it Proof  of Theorem \ref{thm:main2} item d):}\\
By equation (\ref{eqn:summaxs}), if $s\ge 2$
\[\left(\sum_k \max_{n_k<n\le n_{k+1}} n^{r s} |\tau_{\mu}^n(1-\tau_{\mu})^r f |^s \right)^{1/s} \lesssim
 {\bf Q_{2r-1,2,r}}f + {\bf Q_{2(r+1)-1,1,r+1}}f.\]
 Thus the result follows from Theorem \ref{thm:main2} item a) with $s=2$.
\hfill$\square$

\begin{lemma}\label{lem:rn_weak} 
Let $(X,\beta,m)$ be a probability space, $\tau:X\to X$ an ergodic, invertible measure preserving transformation. 
Let $\mu$ be a probability measure on $\mathbb Z$ with property $\text{BA}_1$, and  $s\ge  1$.
Let $\{r_{n+1}\} $ be a sequence with $r_{n+1}-r_n\sim r_n^{\delta}$ for some $0<\delta \le 1$. 
Let $\lambda>0$ and $a$ be a $\lambda$--atom, with $E$ and $E^*$ as in definition \ref{def:atom}.
\begin{enumerate}
\item If $\beta\le r+(1-\delta )(1-1/s)$,
then 
\[ \int_{(E^*)^c} \sum_n r_n^{\beta s}|(\tau_{\mu}^{r_n}-\tau_{\mu}^{r_{n+1}})(I-\tau_{\mu})^r a|^s)^{1/s}\, dm
\lesssim \lambda m(E).
\]
\item If $r>0$ and $\beta\le r-(1-\delta )/s$, 
\[ \int_{(E^*)^c} \sum_n r_n^{\beta s} |\tau_{\mu}^{r_n}(I-\tau_{\mu})^r a|^s)^{1/s}\, dm
\lesssim \lambda m(E).\]
\end{enumerate}
\end{lemma}

\begin{proof}
For simplicity, we only prove case (a) because case (b) follows from similar estimates. 

By Theorem \ref{thm:ergodic_key}, it suffices to check the assumptions of that theorem for \\ $\hat \Delta_n=r_n^{\beta}({\hat\mu}^{r_n}-{\hat\mu}^{r_{n+1}})(1-\hat\mu)^r$ and the sequence $n_k=k$.

From equation (\ref{eq:nk}),  and since $\beta+\delta+(1-\delta )/s\le r+1$,
\begin{align*}
 \left(\sum_{n=1}^{\infty} |\hat\Delta_n(t)|^s\right)^{1/s}  
 \lesssim & \frac {h(t)^{r+1}}{h(t)^{\beta+\delta +(1-\delta )/s} } \lesssim 1.
 \end{align*}
 Hence 
 \[\int_{|t|<.5} \left(\sum_{n=1}^{\infty} |\hat\Delta_n(t)|^s\right)^{1/s}  dt \lesssim 1
 \mbox{ and }
  \sum_{|j|>0} \frac 1{|j|^2} \left(\sum_{n=1}^{\infty} |\hat\Delta_n(1/|j|)|^s\right)^{1/s} <\infty.\]
 By  equation (\ref{eq:nkprime}) 
 \[
 \left(\sum_{n=1}^{\infty} |\hat\Delta'_n(t)|^s\right)^{1/s}  \lesssim \frac{h(t)^r }{h(t)^{\beta+\delta+(1-\delta)/s}} |\hat\mu'(t)|,\]
 and by $\text{BA}_1$(ii), $|t \hat\mu'(t)| \lesssim h(t)$,
\begin{align*} \sum_{|j|>0} \frac 1{|j|^3} \Big(\sum_{n=1}^{\infty}   |\hat\Delta'_n (1 &  /|j|)|^s\Big)^{1/s} 
\lesssim  \sum_{|j|>0} \frac 1{|j|^3}\frac{h(1/|j|)^r}{h(1/|j|)^{\beta+\delta+(1-\delta)/s} } |\hat\mu'(1/|j|)|\\
\lesssim & \sum_{|j|>0} \frac 1{|j|^2}h(1/|j|)^{r+1-(\beta+\delta+(1-\delta)/s)}
\lesssim  \sum_{|j|>0} \frac 1{|j|^2}<\infty.
\end{align*}
 By equation (\ref{eq:nk2primev2}) and $\text{BA}_1$,
\[ \int_{|t|<.5} \left(\sum_{n=1}^{\infty} |\hat\Delta''_n(t)|^s\right)^{1/s}  |t|^2 \, dt 
 \lesssim \int_{|t|<.5}  \frac{h(t)^r}{h(t)^{\beta+\delta+(1-\delta)/s} }\frac{|\hat\mu'(t)|}{|t|} |t|^2 dt 
 \lesssim 1.\]
\end{proof}

The authors are very thankful to the referee for careful reading of our manuscript and for the suggestions for improvement on 
Theorems \ref{thm:RittResult}, 
\ref{thm:main2} 
and Corollary \ref{cor:sup}. 


\end{document}